\documentclass[10pt]{amsart}
\usepackage[cp1251]{inputenc}
\usepackage[english,russian]{babel}
\usepackage{amsmath}
\usepackage{amssymb}
\usepackage{amsfonts}

\newtheorem{lemma}{Lemma}
\newtheorem{theorem}{Theorem}

\newtheorem{corollary}{Corollary}
\newtheorem{proposition}{Proposition}
\newcommand{\PGL}{\mathrm{PGL}}
\newcommand{\GL}{\mathrm{GL}}
\newcommand{\Sym}{\mathrm{Sym}}

\begin{document}

\renewcommand{\refname}{References}
\renewcommand{\proofname}{Proof.}

\title[]{Perfect codes from PGL(2,5) in Star graphs}
\author{{I.Yu.Mogilnykh}}%
\address{Ivan Yurevich Mogilnykh
\newline\hphantom{iii} Novosibirsk State University,
\newline\hphantom{iii} Pirogova st., 1,
\newline\hphantom{iii} 630090, Novosibirsk, Russia
\newline\hphantom{iii} Sobolev Institute of Mathematics,
\newline\hphantom{iii} pr. Koptyuga, 4,
\newline\hphantom{iii} 630090, Novosibirsk, Russia}%

\email{ivmog84@gmail.com}

\thanks{\sc Mogilnykh, I. Yu.,
Perfect codes from PGL(2,5) in Star graphs}
\thanks{\copyright \ 2019 Mogilnykh I.Yu.}
\thanks{\rm  This work was partially supported by the Grant RFBR
18-01-00353A}
\thanks{\it Received  January, 1, 2015, published  March, 1,  2015.}%
\maketitle

\vspace{1cm} \maketitle {\small
\begin{quote}
\noindent{\sc Abstract.} The Star graph $S_n$ is the Cayley graph of the symmetric group
$\Sym_n$ with the generating set $\{(1\mbox{ }i): 2\leq i\leq n \}$.
 Arumugam and Kala proved that $\{\pi\in  \Sym_n: \pi(1)=1\}$ is a perfect code in $S_n$ for any $n, n\geq 3$. In this note we
 show that for any $n, n\geq 6$ the Star graph $S_n$ contains a perfect code
 which is a union of cosets of the embedding of $\PGL(2,5)$ into $\Sym_6$.

\medskip

\noindent{\bf Keywords:} perfect code, efficient dominating set, Cayley graph, Star graph, projective linear group, symmetric group.
 \end{quote}
}

\section{Introduction}

Let $G$ be a group with an inverse-closed generating set $H$ that does not contain the identity.
 {\it The Cayley graph} $\Gamma(G,H)$ is the graph whose
 vertices are the elements of $G$ and the edge set is $\{(hg,g):g \in G,h\in
 H\}$. The  {\it symmetric group} of degree $n$ is denoted by $\Sym_n$. The stabilizer of an element $i\in\{1,\ldots,n\}$ by $\Sym_n$ is denoted by $Stab_i(\Sym_n)$.
{\it The Star graph} $S_n$ is $\Gamma(\Sym_n,\{(1\mbox{ }i): 2\leq i\leq n \})$.

A {\it code} in a graph $G$ is a subset of its vertices. The {\it size} of $C$ is $|C|$.
The {\it minimum distance} of a code is $d=min_{x,y\in C,x\neq y}d(x,y)$, where $d(x,y)$ is the
length of a shortest path connecting $x$ and $y$. A code $C$ is
{\it perfect} (also known as efficient dominating set) in a
$k$-regular graph $\Gamma$ with vertex set $V$ if it has minimum distance 3 and the size of
$C$ attains {\it the Hamming upper bound}, i.e.
$|C|=|V|/(k+1)$. We say that two codes in a graph $\Gamma$ are {\it isomorphic} if there is an automorphism of the graph $\Gamma$  that maps one code into another.

Let $T_0$, $T_1$ be distinct subsets of vertices of a graph $\Gamma$. The ordered pair $(T_0,T_1)$ is called a {\it perfect bitrade}, if for any vertex $x$, the set consisting of $x$ and its neighbors in $\Gamma$ meets $T_0$ and $T_1$ in the same number of vertices  that is zero or one.
The size of $|T_0|$ is called the {\it volume} of the bitrade.
In particular, if $C$ and $C'$ are perfect codes in $\Gamma$, then $(C\setminus C',C'\setminus C)$ is a perfect bitrade. In this case the bitrade $(C\setminus C',C'\setminus C)$ is called {\it embeddable} into a perfect code. 
 In general, bitrades (non necessarily perfect) are often associated with classical combinatorial objects such as perfect codes, Steiner triple and quadruple systems and latin squares  (e.g. see a survey \cite{Kr}).  
 Bitrades are used in constuctions of  the parent combinatorial
objects or for obtaining upper bounds on their number.

The first well-known error-correcting code was the binary Hamming code. This code is a perfect code in the Hamming graph, which is a Cayley graph of the group $Z_2^n$.
Later in \cite{Va} Vasiliev showed that there are perfect codes that are nonisomorphic to the Hamming codes.  A somewhat similar fact holds for the Star graph as in Section 3 we show that there are perfect codes nonisomorphic to the first series of perfect codes in the Star graph from \cite{AK}.

Generally speaking, the permutation codes are subsets of $\Sym_n$ with
 respect to a certain metric. These codes are of practical interest for their various applications in areas such as flash memory storage \cite{JS} and interconnection
 networks \cite{AKrishna}. The permutation codes with the Kendall $\tau$-metric (i.e. codes in the {\it bubble-sort graph} $\Gamma(\Sym_n,\{(i\mbox{ }i+1): 1\leq i\leq n-1
 \})$) were considered by Etzion and Buzaglo in \cite{Et}. They showed that no perfect codes in these graphs exist when $n$ is prime or $4\leq n\leq 10$. In \cite{DT} the nonexistence of the perfect codes in the Cayley graphs $\Gamma(\Sym_n,H)$ was
 established, where $H$ are transpositions that form a tree of
 diameter 3.

The spectral graph theory is important from the point of view of coding theory. 
In particular, according to the famous Lloyd's theorem the
existence of a perfect code in a regular graph necessarily implies
that $-1$ is an eigenvalue of the graph. The integrity of the spectra of several classes of Cayley graphs of the symmetric and the alternating groups was proven in
\cite{R}. The eigenvalues of $S_n$ are all integers $i, -(n-1)\leq
i \leq (n-1)$  that follows from the spectra of the
Jucys-Murphy elements \cite{Ch}. The multiplicities of the eigenvalues of $S_n$ were
studied in \cite{AKK} and the second largest eigenvalue $n-2$
was shown to have multiplicity $(n-1)(n-2)$. In \cite{PI} an
explicit basis for the eigenspace with eigenvalue $n-2$ was found
and a reconstruction property for eigenvectors by its partial
values was proven. Later in \cite{MI} it is shown that the basis consists of eigenvectors with minimum support.

For $l, r \in \Sym_n$ define the following mapping on the vertices of $S_n$: $\lambda_{l,r}(g)=lgr$, $g$ in $\Sym_n$.
\begin{theorem}\label{Feng}\cite{Fe}
The automorphism group of $S_n$ is $\{\lambda_{l,r}:l \in Stab_1(\Sym_n), r\in \Sym_n \}$.
\end{theorem}

In \cite{AK}  Arumugam and Kala showed that $Stab_1(\Sym_n)$ 
is  a perfect code in $S_n$, for any $n\geq 3$ .  Consider the isomorphism class of $Stab_1(\Sym_n)$ in $S_n$.  By Theorem \ref{Feng} the only left  multiplication automorphisms are those by 
the elements from $Stab_1(\Sym_n)$. Therefore we have the following result.

\begin{corollary}
\label{CoroFeng}
The isomorphism class of $Stab_1(\Sym_n)$in $S_n$ is the set of its right cosets in $\Sym_n$.
\end{corollary}


In Section 2 we prove that the projective linear group $\PGL(2,5)$ is  a
perfect code, which is isomorphic to $\{\pi\in
\Sym_6: \pi(1)=1\}$ as a group via an outer automorphism of
$\Sym_6$, but is nonisomorphic to it with respect to the
automorphism group of the Star graph. We continue the study in
Section 3 where we construct a new series of  perfect codes in Star
graphs $S_n$, $n\geq 7$ using cosets of $\PGL(2,5)$. Also we obtain the classification of the isomorphism classes of perfect codes and perfect bitrades in Star graphs $S_n$, $n\leq 6$ by linear programming.

\section{Perfect codes from $\PGL(2,5)$ in $S_6$ }
The action of a group $G$ on a set $M$ is {\it regular} if it is
transitive and $|G|=|M|$, i.e. for any $x$, $y\in M$ there is exactly one element
of $G$ sending $x$ to $y$.

Let $\PGL(n,q)$ be the projective linear group induced by the action of
$\GL(n,q)$ on the 1-dimensional subspaces (projective points) of a
$n$-dimensional space over the field of order $q$. It is well known that $\PGL(n,q)$  acts transitively on the ordered pairs of distinct projective
points for $n\geq 3$ and regularly on the ordered  triples of pairwise distinct projective points when
$n=2$, see e.g. \cite{Dixon}[Exercises 2.8.2 and 2.8.7].

\begin{proposition}\label{PGLprop}
The group $\PGL(2,q)$ acts regularly on the ordered triples of distinct
projective points.
\end{proposition}

In throughout what follows we enumerate the projective points by
the elements of $\{1,\ldots,6\}$, so $\PGL(2,5)$ is embedded in $\Sym_n$,
$n\geq 6$. An element of $\Sym_n$ is a {\it  cycle} of length $m$, if it permutes $i_1,\ldots,i_m\in \{1,\ldots,n\}$ in the cyclic order and fixes every element of $\{1,\ldots,n\}\setminus \{i_1,\ldots,i_m\}.$

\begin{corollary}\label{coro1}
The group $\PGL(2,5)$ does not contain cycles of length $2$ or $3$.
\end{corollary}
\begin{proof}
By Proposition \ref{PGLprop} the group
 $\PGL(2,5)$ is regular on the triples of the elements of
$\{1,\ldots,6\}$. In particular, any permutation of $\PGL(2,5)$ that has at least three fixed projective points is the identity.  We conclude that there are no cycles of length 2 or 3 in $\PGL(2,5)$ since they have three fixed points. 
\end{proof}

\begin{lemma}\label{lem_1}
Let $\pi$ be a permutation from $\Sym_n$, $n\geq 6$. Then $\pi
\PGL(2,5)$ is a code in $S_n$ with the minimum distance 3.
\end{lemma}
\begin{proof}
Suppose that $\pi \pi'$ and $ \pi \pi''$ are adjacent in $S_n$,
$\pi'$, $\pi''\in \PGL(2,5)$. Then by the definition of the Star graph $S_n$ there is $x$, $2 \leq x\leq n$
such that $(1\mbox{ }x)\pi \pi'=\pi \pi''$, so $\pi^{-1}(1\mbox{ }x)\pi=\pi''(\pi')^{-1}$ is in
$\PGL(2,5)$. This contradicts Corollary \ref{coro1} because
$\pi^{-1}(1x)\pi$ is a transposition. If $\pi\pi'$ and $\pi \pi''$
are at distance 2 in $S_n$, then there are $x$ and $y$, $2\leq
x,y\leq n$, $x\neq y$ such that $\pi^{-1}(1\mbox{ }x)(1\mbox{ }y)\pi$ is in
$\PGL(2,5)$.   So, $\pi^{-1}(1\mbox{ }x)(1\mbox{ }y)\pi$  is a cycle of length 3, which contradicts Corollary \ref{coro1}.

\end{proof}


\begin{theorem}
The group $\PGL(2,5)$ is a perfect code in $S_6$ and the partitions of
$\Sym_6$ into the left and into the right cosets  by $\PGL(2,5)$ are  partitions of
the Star graph $S_6$ into perfect codes.
\end{theorem}
\begin{proof} The order of $\PGL(2,5)$ is $5!$, which is the size
of a perfect code in $S_6$ by the Hamming bound. Lemma 1 implies
that $\PGL(2,5)$ as well as any left coset of $\PGL(2,5)$ is a
perfect code. Since  the right multiplication by any element of $S_n$ is an automorphism of $S_n$ by Theorem \ref{Feng}, every right coset of $\PGL(2,5)$ is also a
perfect code. The partitions into the left and
right cosets are different because $\PGL(2,5)$ is not a normal
subgroup in $\Sym_6$.

\end{proof}

\section{Recursive construction for perfect codes in the Star graphs from $\PGL(2,5)$}

Let $C$ be a code in $S_{n}$. For a permutation $\sigma$ from $Sym(n)$ denote by $\sigma C=\{\sigma \pi:\pi \in C\}$. If $\sigma$ fixes 1 by Theorem \ref{Feng} the left multiplication by $\sigma$ is an automorphism of $S_{n}$  and therefore the set of distances between any two permutations of $C$ coincides with that of $\sigma C$.
In this section we show that a code in the Star graph $S_{n-1}$ with minimum distance three could be embedded into a code in the Star graph $S_{n}$ with minimum distance three by taking $(n-1)$
left multiplications of $C$ by transpositions. In particular, we obtain a new infinite series of perfect codes in the Star graphs $S_n$ from $\PGL(2,5)$ for any $n, n\geq 6$.

\begin{theorem}\label{main}
Let $C$ be a code with minimum distance 3 in $S_{n-1}$. Then the code $$C^n=C\cup\bigcup_{2\leq i \leq n-1}(i\mbox{ }n)C$$ is a code of size $|C|(n-1)$ with minimum distance 3.
\end{theorem}
\begin{proof}

We introduce an auxilary notation and prove a technical result.
Let $\Gamma_i$ denote the subgraph of $S_n$ induced by the set of vertices $(i\mbox{ }n)\Sym_{n-1}$, $i\in {1,\ldots,n-1}$, $\Gamma_n$
denote the subgraph of $S_n$ induced by the vertices from $\Sym_{n-1}$.  Note that in \cite{PI} (see also \cite{MI}[Section 6]) a similar partition was considered for constructing a basis for eigenspace of $S_n$ corresponding to eigenvalue $n-2$.

\begin{lemma}\label{L1}
1. For any $i, 2\leq i\leq n$, $\Gamma_i$ is an isometric subgraph of $S_n$ that is isomorphic to $S_{n-1}$.
The set of vertices of $\Gamma_1$ is a perfect code in $S_n$.

2. Let $\pi$ be a permutation from $\Sym_{n-1}$. Then for any  i, $2\leq i\leq n-1$ the vertex $(i\mbox{ }n)\pi$ of $\Gamma_i$ has exactly one neighbor in $S_n$ outside of $\Gamma_i$ and it is the vertex (1\mbox{ }n)(1\mbox{ }i) of $\Gamma_1$.
The only neighbor of $\pi$ in $S_n$ outside $\Gamma_n$ is $(1\mbox{ }n)\pi$.
\end{lemma}
\begin{proof}
1. Obviously, the vertices of $\Sym_{n-1}$ induce an isometric subgraph of $S_n$ which is isomorphic to $S_{n-1}$. 
By Theorem \ref{Feng} the left multiplication by $(i\mbox{ }n)$ is an automorphism of $S_n$ for any $i\in \{2,\ldots,n\}$.
We conclude that $\Gamma_i$ are isomorphic copies of $S_{n-1}$ for any $i\in \{2,\ldots,n\}$. 
By Corollary \ref{CoroFeng} we have that  $(Stab_{1}(\Sym_n))(1\mbox{ } n)=(1\mbox{ } n)\Sym_{n-1}$ is a perfect code in $S_n$. Since this set is exactly the vertices of $\Gamma_1$, we obtain the required. 

2. Since  $\Gamma_i$ is isomorphic to $S_{n-1}$, it is $(n-2)$-regular for $i \in \{2,\ldots,n-1\}$. The remaining neighbor of $(i\mbox{ }n)\pi$ outside $\Gamma_i$ is the vertex $(1\mbox{ }i)(i\mbox{ }n)\pi=(1\mbox{ }n)(1\mbox{ }i)\pi$ of $\Gamma_1$.
\end{proof}

Obviously, the size of $C^n$ is $(n-1)|C|$. We now show that the minimum distance of $C^n$ is three.  We see that each of the graphs $\Gamma_i$ contains the copy $(i\mbox{ }n)C$ of the code $C$, for any $i\in\{2,\ldots ,n-1\}$ and $\Gamma_n$ contains $C$.
The distances between vertices from $(i\mbox{ }n)C$ are the same as those of $C$ in $S_{n-1}$. Therefore, it remains to show that the distances between the vertices of $(i\mbox{ }n)C$ and $(k\mbox{ }n )C$ and the distances  between  the vertices of $(i\mbox{ }n)C$ and $C$ are at least 3, for any distinct $i, k$ such that $2 \leq i,k\leq n-1$. By the second statement of Lemma \ref{L1}, these distances are at least 2.

Let $(i\mbox{ }n)\pi$ and $(k\mbox{ }n)\pi'$ be at distance 2, $\pi,\pi' \in C$. Then by the second statement of Lemma \ref{L1} they both have a common neighbor in $\Gamma_1$, which is $(1\mbox{ }n)(1\mbox{ }i)\pi=(1\mbox{ }n)(1\mbox{ }k)\pi'$.
This implies that $(1\mbox{ }i)(1\mbox{ }k)\pi'=\pi$ for $1\leq i,k \leq n-1$, or equivalently $\pi$ and $\pi'$ are at distance 2 in $S_{n-1}$. This contradicts the minimum distance of $C$.

Let $(i\mbox{ }n)\pi$ and $\pi'$ be at distance 2, $\pi,\pi' \in C$. By the second statement of Lemma \ref{L1} the only neighbor of $(i\mbox{ }\pi)$ outside of $\Gamma_i$ is $(1\mbox{ }n)(1\mbox{ }i)\pi$ and the only neighbor of $\pi'$ outside $\Gamma_n$  is $(1\mbox{ }n)\pi'$. So we see that  $(1\mbox{ }n)(1\mbox{ }i)\pi=(1\mbox{ }n)\pi'$, which contradicts 
the minimum distance of $C$.

\end{proof}

\begin{corollary}
For any $n\geq 6$ there is a perfect code in $S_n$ which is not isomorphic to $Stab_1(Sym_n)$.
\end{corollary}
\begin{proof}
Consider the code $D$ which is obtained by iteratively applying construction from Theorem \ref{main} $(n-6)$ times to the code $\PGL(2,5)$. By the construction, the code $\PGL(2,5)$ is a subcode of $D$.
Proposition \ref{PGLprop} implies that there are permutations $\pi, \pi'$ in $\PGL(2,5)$ such that $\pi(1)\neq \pi'(1)$.  By Corollary \ref{CoroFeng} the isomorphism class of  $Stab_1(Sym_n)$ in $S_n$ consists of its right cosets. Since we have that $\pi(1)=\pi'(1)$ for any $\pi$ and $\pi'$ from a right coset of $Stab_1(Sym_n)$, we conclude that $D$ is not isomorphic to $Stab_1(Sym_n)$. 
\end{proof}

We proceed with the following computational results for small Star graphs.

\begin{proposition}
1. The isomorphism class of $Stab_1(Sym_n)$ is the only isomorphism class of the perfect codes in $S_n$ for n=3,4,5.  \\
2. The isomorphism classes of $Stab_1(Sym_6)$ and $\PGL(2,5)$ are the only isomorphism classes of the perfect codes in $S_6$.
\end{proposition}

\begin{proof}
 For  $n=3$ and $4$ the uniqueness of perfect code in $S_n$ could be shown by hand.
In case when $n=5$  and $6$ the result was obtained by binary linear programming. Because $S_n$ is a transitive graph, without restriction of generality, we can consider the  perfect codes containing the identity permutation.
In case $n=5$ there is one solution to the binary linear programming problem, which is  $Stab_1(Sym_n)$. 

Let $n$ be six. We consider any transposition that preserves 1, say $(2\mbox{ }3)$. By the definition of the Star graph, $(2\mbox{ }3)$ is at distance three from the identity permutation.
Now we split the set of all codes as follows: the codes that contain the permutation  $(2\mbox{ }3)$ and those that do not. We then solve two linear programming problems separately for these cases. There are 6 solutions (perfect codes) that does not contain $(2\mbox{ }3)$. These are $\PGL(2,5)$ and its five conjugations. When  $(2\mbox{ }3)$ is in the code, the returns with the only solution which is $Stab_1(Sym_n)$.

\end{proof}

\begin{proposition}
All perfect bitrades in $S_n$ are embeddable for $3\leq n\leq 6$. For $n\in\{3,4,5\}$ their volumes are equal to $(n-1)!$. For $n=6$ the volumes of bitrades are $120$, $100$ and $96$. 

\end{proposition}
\begin{proof}
The statement is obvious for $n=3$. Using linear programming approach by PC we found that for $n=4, 5, 6$ all bitrades are embeddable and have the corresponding volumes.
 When $n$ is $6$,   a perfect bitrade $(C\setminus C',C'\setminus C)$  has volume $120$ if 
$C$ and $C'$ are disjoint perfect codes, e.g. $Stab_1(\Sym_6)$ and $Stab_1(\Sym_6)(1\mbox{ }6)$.
By Proposition \ref{PGLprop} the group $\PGL(2,5)$ acts transitively on the set $\{1,\ldots,6\}$, so there are exactly $20$ permutations from $\PGL(2,5)$ that fix $1$.
So we see that a perfect bitrade $(C\setminus C',C'\setminus C)$ is of volume $100$ if $C$ is $Stab_1(\Sym_6)$  and $C'$ is $\PGL(2,5)$.
Finally, $(C\setminus C',C'\setminus C)$ is a perfect bitrade
of volume $96$ if $C$ is $\PGL(2,5)$ and $C'$ is one of its nontrivial conjugations. Indeed, $\PGL(2,5)$ is isomorphic to $\Sym_5$  via an outer automorphism of $\Sym_6$. 
Therefore the intersection of $\PGL(2,5)$  and its conjugation is a subgroup which is isomorphic to the intersection $\Sym_5$ and some of its conjugation $Stab_i(\Sym_6)$, $i \in \{1,\ldots,5\}$. Since the latter intersection is of order 4!=24, the proposition is true.

\end{proof}

{\bf Acknowledgements} The author gradually thanks Sergey Avgustinovich for introducing the problem and stimulating discussions.

\end{document}